\PassOptionsToPackage{dvipsnames}{xcolor}
\documentclass[11pt]{amsart}
\usepackage{amsmath,amssymb,amsfonts}

\author[D.\ Chan]{David Chan}
\address[Chan]{
Department of Mathematics,
         Michigan State University,
         East Lansing, MI,
         U.S.A.}
\email{\href{mailto:chandav2@msu.edu}{chandav2@msu.edu}}

\author[B.\ Spitz]{Ben Spitz}
\address[Spitz]{Department of Mathematics, Indiana University, Bloomington, IN, U.S.A.}
\email{\href{mailto:bespitz@iu.edu}{bespitz@iu.edu}}



\usepackage[margin=1in]{geometry}

\usepackage{quiver} 
\usepackage{mathrsfs} 
\usepackage[shortlabels]{enumitem} 
\usepackage{microtype} 

\usepackage[unicode=true, pdfusetitle,
    colorlinks=true,
    linkcolor=orange,
    citecolor=orange,
    urlcolor=orange
]{hyperref}

\usepackage[nameinlink]{cleveref}

\DeclareMathOperator{\res}{res}
\DeclareMathOperator{\tr}{tr}
\DeclareMathOperator{\nm}{nm}
\DeclareMathOperator{\conj}{c}
\DeclareMathOperator{\Spec}{Spec}
\newcommand{\uA}{\underline{A}} 
\DeclareMathOperator{\FP}{FP}
\newcommand{\mf}[1]{\mathfrak{#1}}

\newcommand{\ZZ}{\mathbb{Z}}

\newcommand{\I}{\mathcal{I}}
\newcommand{\cO}{\mathcal{O}}

\newcommand{\Q}{\mathtt{Q}} 
\DeclareMathOperator{\nil}{Nil}
\DeclareMathOperator{\kil}{Kil}
\DeclareMathOperator{\id}{id}

\numberwithin{equation}{section}
\crefname{lemma}{Lemma}{Lemmas}
\crefname{theorem}{Theorem}{Theorems}
\crefname{lettertheorem}{Theorem}{Theorems}
\crefname{definition}{Definition}{Definitions}
\crefname{proposition}{Proposition}{Propositions}
\crefname{remark}{Remark}{Remarks}
\crefname{corollary}{Corollary}{Corollaries}
\crefname{section}{Section}{Sections}

\theoremstyle{plain}
\newtheorem{theorem}[equation]{Theorem}
\newtheorem{lettertheorem}{Theorem}

\newtheorem{corollary}[equation]{Corollary}
\newtheorem{proposition}[equation]{Proposition}
\newtheorem{lemma}[equation]{Lemma}

\theoremstyle{definition}
\newtheorem{definition}[equation]{Definition}
\newtheorem{example}[equation]{Example}
\newtheorem{remark}[equation]{Remark}
\newtheorem{notation}[equation]{Notation}

\usepackage[
    maxbibnames=99,
    isbn=false,
    url=false,
    doi=true,
    backref=true
]{biblatex}
\DefineBibliographyStrings{english}{%
    backrefpage = {Cited on page},
    backrefpages = {Cited on pages},
}

\usepackage[
    textwidth=3cm,
    colorinlistoftodos]
{todonotes} 

\title{Radicals and Nilpotents in Equivariant Algebra}

\begin{document}

\begin{abstract}
    Associated to each Tambara functor $T$ is its \emph{Nakaoka spectrum} $\Spec(T)$, analogous to the Zariski spectrum of a commutative ring. We establish that this topological space is spectral. This result follows from an analysis of the notion of nilpotence in Tambara functors. We prove that the \emph{nilradical} of a Tambara functor $T$ (the intersection of all of its prime ideals) is computed levelwise, i.e., consists precisely of the nilpotent elements in $T$. In contrast to ordinary commutative algebra, the nilpotents of $T$ are not the same as the elements $x$ such that $T[1/x] = 0$; we therefore also give a classification of these elements. As a corollary, we observe that the set of these elements in $\pi_\star^s$ (the equivariant stable stems, viewed as an $\mathrm{RO}(G)$-graded Tambara functor) forms an ideal.
\end{abstract}

\maketitle


\section{Introduction}

A generalization of commutative rings, Tambara functors model multiplicative structures which arise in representation theory and equivariant homotopy theory \cite{tambara:1993,HHR16,Brun}. There is a robust analogy between the algebraic theory of Tambara functors and classical commutative algebra \cite{nakaoka:2012,hill:2017,LRZ,AngeltveitBohmann,HillMehrleQuigley}.  In this paper, we build on one aspect of this connection by studying the notion of nilpotence in Tambara functors. Just as in the classical setting, we show that nilpotents are precisely those elements which are in every prime ideal. On the other hand, we also observe that they are not the set of elements which kill the Tambara functor upon being inverted.

To explain these results more precisely, we need a few definitions.  Given a finite group $G$, a $G$-Tambara functor $T$ consists of the data of commutative rings $T(G/H)$ for each subgroup $H\leq G$ (the \emph{levels} of the Tambara functor $T$), together with various structure morphisms connecting these rings.  A standard example is obtained from any $G$-Galois field extension $L/K$, where $T(G/H) := L^H$, the fixed point field, and the structure morphisms encode the Galois-theoretic norm and trace maps.

An \textit{ideal} of a Tambara functor is a collection of ideals $I(G/H) \subseteq T(G/H)$ (one for each subgroup $H \leq G$) which are closed under the structure maps.  In analogy with commutative algebra, an ideal $\mf{p}$ is prime if $IJ\subseteq\mf{p}$ implies either $I\subseteq\mf{p}$ or $J\subseteq\mf{p}$.  It is important to note that this product of ideals $IJ$ is not necessarily computed levelwise, and thus the condition of being a prime Tambara ideal does not imply that each ideal $\mf{p}(G/H) \subseteq T(G/H)$ is a prime ideal.  Moreover, it is possible for there to be prime ideals $\mf{a} \subseteq T(G/H)$ which do not appear as $\mf{p}(G/H)$ for any prime ideal $\mf{p}\subseteq T$.  Thus, despite the fact that the definition of prime Tambara ideals mimics the definition of prime ideals in commutative rings exactly, the theory of prime Tambara ideals is somewhat idiosyncratic. Computations of $\Spec(T)$ for some Tambara functors have been carried out in \cite{nakaoka:2014,CalleGinnett:23,CMQSV:24,CCMQSV}.

In a commutative ring $R$, an element is nilpotent if it satisfies the following three equivalent conditions:
\begin{enumerate}
    \item there is a natural number $n$ so that $x^n=0$;
    \item $x$ is contained in every prime ideal of $R$;
    \item the localization $R[1/x]$ is the zero ring.
\end{enumerate}
In equivariant algebra, one finds in many examples (see \Cref{example: bad kilpotents} below, for instance) that $(3)$ is a strictly weaker condition.  Specifically, there are non-nilpotent elements $x$ in Tambara functors such that the only map which inverts them is the terminal map to the zero Tambara functor.

Indeed, one can already see the problem at the level of rings with $G$-action.  Let $R$ be a (non-zero) commutative ring and let $\mathrm{Map}(G,R) = \prod_{g\in G} R$ be the ring of functions from $G$ to $R$. This ring has a $G$-action by permuting the coordinates.  If $\chi_g\colon G\to R$ is the indicator function of $g\in G$ then $\chi_g$ is not nilpotent.  However, observe that for any $1\neq h\in G$ we have $\chi_g\cdot (h\chi_g)=\chi_g\cdot \chi_{gh}=0$, and thus any \emph{equivariant} ring map $\mathrm{Map}(G,R)\to S$ which inverts $\chi_g$ must invert $0$ and hence $S=0$.  As Tambara functors are a generalization of $G$-rings, such examples are commonplace.

Our first main result is that, surprisingly, (1) and (2) remain equivalent for Tambara functors.

\begin{lettertheorem}[{\cref{thm: main}}]\label{theorem: main theorem 1 intro}
    The intersection of all prime Tambara ideals in a Tambara functor is the collection of nilpotent elements.
\end{lettertheorem}

To highlight why this result is surprising, we recall the standard proof of this fact for commutative rings: for any element $x$ in a commutative ring $R$, we have biconditionals
\begin{align*}
    x \text{ is nilpotent} & \iff R[1/x] = 0                                                                                                           \\
                           & \iff \Spec(R[1/x]) = \varnothing                                                                                          \\
                           & \iff \{\mathfrak{p} \in \Spec(R) : x \notin \mathfrak{p}\} = \operatorname{img}(\Spec(R[1/x])\to \Spec(R)) = \varnothing.
\end{align*}
For Tambara functors, this argument has no chance of going through as written, as it relies on the equivalence of conditions $(1)$ and $(3)$ above which, as mentioned, fails in this setting.  Nevertheless, the intersection of all prime ideals is the collection of nilpotent elements.

We remark that a construction of radicals of ideals for Tambara functors also appears in work of Nakaoka \cite{nakaoka:2012}.  A consequence of our work is that Nakaoka's construction of radical ideals is equivalent to the much simpler construction of taking levelwise radicals; see \cref{cor: nakaoka radicals}.

The collection of prime ideals in a Tambara functor $T$ assembles into a topological space called the \emph{Nakaoka spectrum of $T$}, which we denote by $\Spec(T)$.  In prior work with collaborators, the authors have shown that $\Spec(T)$ is a sober, quasi-compact topological space \cite[Theorem A]{CMQSV:24}.  As an application of \cref{theorem: main theorem 1 intro}, we improve upon this result.  Recall that a topological space $X$ is \emph{spectral} if there exists a commutative ring $R$ such that $X$ is homeomorphic to $\Spec(R)$.
\begin{lettertheorem}[{\cref{cor: spectral}}]\label{theorem: main theorem 2 intro}
    The Nakaoka spectrum of any Tambara functor is a spectral space.
\end{lettertheorem}

Finally, we turn to the question of characterizing elements in a Tambara functor which satisfy property (3); that is, we study those elements $x$ such that $T[1/x]=0$.  To differentiate these from nilpotents, we call such elements \emph{kilpotents}, as they kill the Tambara functor upon inversion.  While these elements do not admit a description as those in the intersection of some class of ideals (indeed, the collection of kilpotents is not always an ideal), these elements are important to keep track of, as localization is such a central tool in applications, especially stable homotopy theory; see, for instance,  \cite{Hill:Localization,HillHopkins:closure}.

\begin{lettertheorem}[{\cref{thm: kilradical}}]\label{theorem: main theorem 3 intro}
    Let $T$ be a $G$-Tambara functor, and let $x \in T(G/H)$ be some element of $T$. Then $T[1/x] = 0$ if and only if
    \[\prod_{g \in G} g \cdot \res^H_e(x)\]
    is nilpotent.
\end{lettertheorem}

As a consequence, one can show that the set of kilpotents in a Tambara functor \emph{does} form an ideal in certain examples, e.g., the set of kilpotents in the Burnside Tambara functor is an ideal; see \cref{cor:  trivial action means kilradical is an ideal,cor: kil of A}. The analogous claim is true for the $\mathrm{RO}(G)$-graded Tambara functor of equivariant stable stems; see \Cref{stable stems}.

\subsection{Outline}

Background on Tambara functors and ideals is given in \cref{sec: background}.  We prove some technical results about prime ideals in \cref{sec:pre-ideals}.  The proof of \cref{theorem: main theorem 1 intro} is given in \cref{sec: nilpotents}.  The proof of \cref{theorem: main theorem 2 intro} is given in \cref{sec: spectral spaces}.  Finally, \cref{sec: kilradical} introduces the kilradical of a Tambara functor and proves \cref{theorem: main theorem 3 intro}.

\subsection{Conventions}
Throughout the paper we will use $G$ to denote a finite group.  We will write $H\leq G$ to mean that $H$ is a subgroup of $G$, and $H<G$ to mean that $H$ is a proper subgroup. Similar notations apply to subsets ($\subseteq$) and proper subsets ($\subset$).

\subsection{Acknowledgements}

We would like to thank Maxine Calle and Teena Gerhardt for helpful feedback on a draft of this paper. We would also like to thank Phil, for years of inspiration. DC was partially supported by NSF grant DMS-2135960. The authors thank anonymous referees for helpful comments which improved the paper, especially the suggestion to include pre-ideals and annihilators.

\section{Background}\label{sec: background}

This section contains background material used throughout the rest of the paper.  We first introduce Tambara functors, then give the definitions of ideals.  The section ends with a discussion of prime ideals and the Nakaoka spectrum of a Tambara functor.

\subsection{Tambara functors}

Here we give a very brief introduction to the theory of Tambara functors. We refer the reader to \cite{tambara:1993} for the original definition, and \cite{strickland:2012} for a thorough treatment.

Let $G$ be a finite group. Underlying the definition of $G$-Tambara functors is the \emph{polynomial category} $\mathcal{P}_G$ (also known as the \emph{bispan category}). The objects of $\mathcal{P}_G$ are finite $G$-sets. Given two finite $G$-sets $X$ and $Y$, the hom-set $\mathcal{P}_G(X,Y)$ is the set of isomorphism classes of diagrams
\[X \leftarrow A \to B \to Y\]
(called \emph{bispans}) in the category $\mathsf{set}_G$ of finite $G$-sets. Two such bispans $X \leftarrow A \to B \to Y$ and $X \leftarrow A' \to B' \to Y$ are isomorphic if there exist isomorphisms $A \to A'$ and $B \to B'$ (in $\mathsf{set}_G$) making the diagram
\[\begin{tikzcd}[ampersand replacement=\&,row sep=tiny]
        \& A \& B \\
        X \&\&\& Y \\
        \& {A'} \& {B'}
        \arrow[from=1-2, to=1-3]
        \arrow[from=1-2, to=2-1]
        \arrow["\cong"{description}, from=1-2, to=3-2]
        \arrow[from=1-3, to=2-4]
        \arrow["\cong"{description}, from=1-3, to=3-3]
        \arrow[from=3-2, to=2-1]
        \arrow[from=3-2, to=3-3]
        \arrow[from=3-3, to=2-4]
    \end{tikzcd}\]
commute. We denote the isomorphism class of a bispan $X \leftarrow A \to B \to Y$ by $[X \leftarrow A \to B \to Y]$.

For each morphism $f \colon X \to Y$ in $\mathsf{set}_G$, we obtain three distinguished morphisms in $\mathcal{P}_G$:
\begin{align*}
    T_f & := [X \xleftarrow{\id} X \xrightarrow{\id} X \xrightarrow{f} Y] \in \mathcal{P}_G(X,Y), \\
    N_f & := [X \xleftarrow{\id} X \xrightarrow{f} Y \xrightarrow{\id} Y] \in \mathcal{P}_G(X,Y), \\
    R_f & := [Y \xleftarrow{f} X \xrightarrow{\id} X \xrightarrow{\id} X] \in \mathcal{P}_G(Y,X).
\end{align*}

The composition operation of the category $\mathcal{P}_G$ is determined by the following rules:
\begin{enumerate}
    \item For any bispan $X \xleftarrow{h} A \xrightarrow{g} B \xrightarrow{f} Y$, \[T_f \circ N_g \circ R_h = [X \xleftarrow{h} A \xrightarrow{g} B \xrightarrow{f} Y].\]
    \item For any morphisms $f \colon Y \to Z$ and $g \colon X \to Y$ in $\mathsf{set}_G$:
          \begin{enumerate}
              \item $T_f \circ T_g = T_{f \circ g}$,
              \item $N_f \circ N_g = N_{f \circ g}$,
              \item $R_g \circ R_f = R_{f \circ g}$.
          \end{enumerate}
    \item For any morphisms $f \colon X \to Y$ and $g \colon X' \to Y$ in $\mathsf{set}_G$,
          \[R_f \circ T_g = T_{g'} \circ R_{f'}\]
          and
          \[R_f \circ N_g = N_{g'} \circ R_{f'},\]
          where $f'$ and $g'$ are obtained from the pullback diagram
          \[\begin{tikzcd}[ampersand replacement=\&]
                  {X \times_Y X'} \& {X'} \\
                  X \& Y
                  \arrow["{f'}", from=1-1, to=1-2]
                  \arrow["{g'}"', from=1-1, to=2-1]
                  \arrow["\lrcorner"{anchor=center, pos=0.125}, draw=none, from=1-1, to=2-2]
                  \arrow["g", from=1-2, to=2-2]
                  \arrow["f"', from=2-1, to=2-2]
              \end{tikzcd}\]
    \item For any morphisms $f \colon Y \to Z$ and $g \colon X \to Y$ in $\mathsf{set}_G$,
          \[N_f \circ T_g = T_{\Pi_f g} \circ N_{f'} \circ R_{\varepsilon}\]
          where $\Pi_f \colon \mathsf{set}_G/Y \to \mathsf{set}_G/Z$ is the \emph{dependent product} functor (right adjoint to the pullback functor $f^* \colon \mathsf{set}_G/Z \to \mathsf{set}_G/Y$), $f'$ is the pullback of $f$ along $\Pi_f g$, and $\varepsilon$ is the counit of the adjunction $f^* \dashv \Pi_f$. In other words, these maps form an \emph{exponential diagram}
          \[\begin{tikzcd}[ampersand replacement=\&]
                  \& \bullet \& \bullet \\
                  X \& Y \& Z
                  \arrow["{f'}", from=1-2, to=1-3]
                  \arrow["\varepsilon"', from=1-2, to=2-1]
                  \arrow[from=1-2, to=2-2]
                  \arrow["\lrcorner"{anchor=center, pos=0.125}, draw=none, from=1-2, to=2-3]
                  \arrow["{\Pi_f g}", from=1-3, to=2-3]
                  \arrow["g"', from=2-1, to=2-2]
                  \arrow["f"', from=2-2, to=2-3]
              \end{tikzcd}.\]
\end{enumerate}

The first rule states that every morphism in $\mathcal{P}_G$ is equal to one of the form $T \circ N \circ R$, i.e., the distinguished morphisms of types $T$, $N$, and $R$ generate the category $\mathcal{P}_G$. The second rule describes how to compose two generating morphisms of the same type. The third rule describes how to rewrite a composite $R \circ T$ in the form $T \circ R$, or a composite $R \circ N$ in the form $N \circ R$. The fourth rule describes how to rewrite a composite $N \circ T$ in the form $T \circ N \circ R$. Overall, this gives a procedure for writing any composite of two morphisms in the form $T \circ N \circ R$:
\begin{align*}
    T \circ N \circ {\color{blue} R \circ T} \circ N \circ R & \overset{(3)}{\rightsquigarrow} T \circ {\color{blue} N \circ T} \circ R \circ N \circ R                                           \\
                                                             & \overset{(4)}{\rightsquigarrow} T \circ T \circ N \circ R \circ {\color{blue} R \circ N} \circ R                                   \\
                                                             & \overset{(3)}{\rightsquigarrow} T \circ T \circ N \circ {\color{blue} R \circ N} \circ R \circ R                                   \\
                                                             & \overset{(3)}{\rightsquigarrow} {\color{red} T \circ T} \circ {\color{Rhodamine} N \circ N} \circ {\color{blue} R \circ R \circ R} \\
                                                             & \overset{(2)}{\rightsquigarrow} T \circ N \circ R,
\end{align*}
and thus the composition operation in $\mathcal{P}_G$ is fully determined by these rules.

An essential property of the category $\mathcal{P}_G$ is that it admits all finite products. The terminal object of $\mathcal{P}_G$ is $\varnothing$, and the product of two objects $X$ and $Y$ in $\mathcal{P}_G$ is the disjoint union $X \amalg Y$ (i.e., the coproduct of $X$ and $Y$ in $\mathsf{set}_G$). For each finite $G$-set $X$, the natural maps $!_X \colon \varnothing \to X$ and $\nabla_X \colon X \amalg X \to X$ in $\mathsf{set}_G$ induce morphisms
\begin{align*}
    0_X      & := T_{!_X} \in \mathcal{P}_G(\varnothing, X),     \\
    1_X      & := N_{!_X} \in \mathcal{P}_G(\varnothing, X),     \\
    \alpha_X & := T_{\nabla_X} \in \mathcal{P}_G(X \amalg X, X), \\
    \mu_X    & := N_{\nabla_X} \in \mathcal{P}_G(X \amalg X, X)
\end{align*}
in $\mathcal{P}_G$, giving $X$ the structure of a commutative semiring\footnote{A semiring is exactly the same as a ring, except that it need not admit additive inverses. All rings and semirings in this article are unital.} object (with addition $\alpha_X$, multiplication $\mu_X$, additive identity $0_X$, and multiplicative identity $1_X$).

\begin{definition}
    A \emph{$G$-semi-Tambara functor} is a product-preserving functor from $\mathcal{P}_G$ to the category of (all) sets.
\end{definition}

Since product-preserving functors preserve commutative semiring objects, we have that the set $S(X)$ is naturally endowed with the structure of a commutative semiring for all $G$-semi-Tambara functors $S$ and all finite $G$-sets $X$.

\begin{definition}
    A \emph{$G$-Tambara functor} is a $G$-semi-Tambara functor $S$ such that $S(X)$ is in fact a commutative ring for all finite $G$-sets $X$. The category $\mathsf{Tamb}_G$ of $G$-Tambara functors has as objects all $G$-Tambara functors and as morphisms all natural transformations between $G$-Tambara functors.
\end{definition}

\begin{notation}
    When working with a $G$-(semi-)Tambara functor $S$, given a morphism $f \colon X \to Y$ in $\mathsf{set}_G$, one writes $\tr_f, \nm_f, \res_f$ for the functions $S(T_f), S(N_f), S(R_f)$. These operations are called \emph{transfer}, \emph{norm}, and \emph{restriction}, respectively.
\end{notation}

\begin{proposition}\label{prop: basic facts about TNR}
    Let $f\colon X \to Y$ be a morphism in $\mathsf{set}_G$. For any $G$-Tambara functor $S$:
    \begin{enumerate}
        \item $\tr_f \colon S(X) \to S(Y)$ is a homomorphism of underlying additive groups,
        \item $\nm_f \colon S(X) \to S(Y)$ is a homomorphism of underlying multiplicative monoids, and when $f$ is surjective $\nm_f(0)=0$,
        \item $\res_f \colon S(Y) \to S(X)$ is a ring homomorphism.
    \end{enumerate}
\end{proposition}

If $g \colon X \to Y$ is an \emph{isomorphism} in $\mathsf{set}_G$, one has
\[T_g = N_g = R_{g^{-1}}\]
in $\mathcal{P}_G$, and thus $\tr_g = \nm_g = \res_{g^{-1}}$. Moreover, $R_{\id_X} = T_{\id_X} = N_{\id_X}$ is an identity morphism in $\mathcal{P}_G$. In particular, one obtains a canonical left action of $\operatorname{Aut}_{\mathsf{set}_G}(X)$ on $S(X)$, called the \emph{Weyl action}.
\begin{example}
    When $X= G/H$ for some subgroup $H\leq G$ there is a canonical identification $\operatorname{Aut}_{\mathsf{set}_G}(G/H)\cong N_G(H)/H$, where $N_G(H)$ is the normalizer of $H$ in $G$.  In particular, when $H=e$ we have $\operatorname{Aut}_{\mathsf{set}_G}(G/e)\cong G$, and $T(G/e)$ is a ring with $G$-action.
\end{example}

In the definition of Tambara functors given above, one must provide infinitely many data to specify a $G$-Tambara functor (e.g., a commutative ring for \emph{each} finite $G$-set). However, it turns out that a Tambara functor is determined by a finite subcollection of these data. Recall that, by definition, a $G$-(semi-)Tambara functor is a product-preserving functor $\mathcal{P}_G \to \mathsf{Set}$, and that the product in $\mathcal{P}_G$ is given on objects by disjoint union. Since each finite $G$-set decomposes as the disjoint union of its orbits, and every orbit in a finite $G$-set is isomorphic to $G/H$ for some subgroup $H \leq G$, we see that the value of any $G$-(semi-)Tambara functor $S$ on a finite $G$-set $X$ is determined (up to isomorphism) by the values $S(G/H)$ as $H$ ranges over the subgroups of $G$. One can also carry out this decomposition argument for morphisms.

For a containment of subgroups $H \leq K$, let $q \colon G/H \to G/K$ be the canonical quotient $q(xH) = xK$. One writes $\tr^K_H \colon S(G/H) \to S(G/K)$ as shorthand for the transfer along $q$, and likewise for $\nm^K_H \colon S(G/H) \to S(G/K)$ and $\res^K_H \colon S(G/K) \to S(G/H)$. For a subgroup $H$ and an element $g \in G$, one writes $\conj_{g,H} \colon S(G/H) \to S(G/gHg^{-1})$ (or simply $\conj_g$ when the subgroup $H$ is clear from context) for the norm (or transfer) along the conjugation isomorphism $H\to gHg^{-1}$. Define a \emph{crude Tambara functor} to be a collection of sets $\{T(G/H)\}_{H\leq G}$  together with set maps $\tr^K_H$, $\nm^K_H$, $\res^K_H$, and $\conj_g$ as above with no further requirements.  There is a forgetful functor from Tambara functors to crude Tambara functors which simply forgets the values of $T(X)$ for $X\neq G/H$.

\begin{proposition}
    The forgetful functor from Tambara functors to crude Tambara functors is fully faithful, hence also conservative.
\end{proposition}

Explicitly, this means that any Tambara functor is completely determined up to unique isomorphism by the values of $T(G/H)$, and the maps $\tr^K_H$, $\nm^K_H$, $\res^K_H$, and $\conj_g$ for $H\leq K\leq G$ and $g\in G$. As a consequence, morphisms of Tambara functors are also determined by finitely many data.

The next two examples unpack two important relations between the transfer, norm, and restriction operations.
\begin{example}[Double-coset formulas]\label{example: double coset}
    Let $H,L\leq K$ be subgroups of a group $G$. Consider a pullback square
    \[
        \begin{tikzcd}
            P \ar[d,"p" left] \ar[r,"q"] & G/H \ar[d]\\
            G/L\ar[r] & G/K
        \end{tikzcd}
    \]
    in the category of finite $G$-sets, where the unlabeled maps are the canonical quotient maps.  In this case, one can show that $P$ is isomorphic, as a $G$-set, to the disjoint union
    \[\coprod_{\gamma} G/(L\cap \gamma H\gamma^{-1})\]
    where $\gamma$ ranges over a set of representatives of the double cosets $L \backslash K/H$.  After applying a Tambara functor, this pullback diagram yields the additive and multiplicative \emph{double coset formulas}:
    \begin{align*}
        \res^K_L\tr^K_H & = \sum\limits_{\gamma\in L\backslash K/H} \tr^L_{L\cap \gamma H\gamma^{-1}}\conj_{\gamma} \res^{H}_{\gamma^{-1}L\gamma\cap H}  \\
        \res^K_L\nm^K_H & = \prod\limits_{\gamma\in L\backslash K/H} \nm^L_{L\cap \gamma H\gamma^{-1}}\conj_{\gamma} \res^{H}_{\gamma^{-1}L\gamma\cap H}
    \end{align*}
    which play a significant role in computations.
\end{example}

\begin{example}[Frobenius reciprocity]
    Let $H\leq K$ be subgroups of $G$ and let $q\colon G/H\to G/K$ be the canonical quotient map.  There is an exponential diagram
    \[\begin{tikzcd}
            && {G/H\amalg G/H} && {G/H} \\
            {G/H\amalg G/K} && {G/K\amalg G/K} && {G/K}
            \arrow["\nabla", from=1-3, to=1-5]
            \arrow["{\id \amalg q}"', from=1-3, to=2-1]
            \arrow["{q\amalg q}"', from=1-3, to=2-3]
            \arrow["q", from=1-5, to=2-5]
            \arrow["{q\amalg \id}"', from=2-1, to=2-3]
            \arrow["\nabla"', from=2-3, to=2-5]
        \end{tikzcd}\]
    which encodes the relation $\tr_H^K(x)\cdot y = \tr^K_H(x\cdot \res^K_H(y))$ for any $x\in T(G/H)$ and $y\in T(G/K)$ (where $T$ is any Tambara functor). This relation is called \emph{Frobenius reciprocity}.
\end{example}

\begin{example}\label{example FP}
    Let $R$ be a commutative ring with $G$-action. The \emph{fixed point Tambara functor} $\FP(R)$ is given by $\FP(R)(X) = \mathrm{Map}_G(X,R)$, the set of equivariant functions from $X$ to $R$. The name comes from the observation that there is an isomorphism of rings $\mathrm{Map}_G(G/H,R)\xrightarrow{\cong} R^H$, given by $\varphi\mapsto \varphi(H)$, for any subgroup $H\leq G$.

    For a map of finite $G$-sets $f\colon X\to Y$ the restriction along $f$ is given by $\res_f = f^*\colon \mathrm{Map}_G(Y,R)\to \mathrm{Map}_G(X,R)$.  The transfer and norm $\tr_f,\nm_f\colon \mathrm{Map}_G(X,R)\to \mathrm{Map}_G(Y,R)$ are given by the formulae
    \begin{align*}
        \tr_f(\varphi)(y) & = \sum_{x\in f^{-1}(y)} \varphi(x)   \\
        \nm_f(\varphi)(y) & = \prod_{x\in f^{-1}(y)} \varphi(x).
    \end{align*}
    The construction $R\mapsto \FP(R)$ is functorial in equivariant ring maps and is right adjoint to the functor $T\mapsto T(G/e)$ from $G$-Tambara functors to commutative rings with $G$-action. The unit map $T\to \FP(T(G/e))$ is completely determined by the fact that it is the identity on level $G/e$.
\end{example}

\begin{example}\label{example: Burnside}
    The \emph{Burnside ring} $A(G)$ of a finite group $G$ is the group completion of the commutative monoid whose elements are isomorphism classes of finite $G$-sets and with addition given by disjoint union. In other words, the elements of $A(G)$ are formal differences of isomorphism classes of $G$-sets. The multiplication is induced by cartesian product of finite $G$-sets.  The assignment $G/H\mapsto A(H)$ defines a Tambara functor called the \emph{Burnside Tambara functor}, denoted $\uA$.  The restriction maps $\res^K_H\colon A(K)\to A(H)$ are obtained by restriction of $K$-action to $H$-action. The transfers and norms $\tr^K_H,\nm^K_H\colon A(H)\to A(K)$ are induced by induction and coinduction, respectively.  The conjugations $\conj_{g,H} \colon A(H)\to A(gHg^{-1})$ are induced by pushing forward along the isomorphism of categories $\mathrm{set}_H\to \mathrm{set}_{gHg^{-1}}$ induced by the conjugation isomorphism $H \cong gHg^{-1}$.

    The Burnside Tambara functor is the initial Tambara functor: for any other Tambara functor $T$ there is a unique map $\uA \to T$.  Thus, one should think of the Burnside Tambara functor as the analogue of the ring of integers in this context.
\end{example}

We end this section with a basic fact from the theory of Tambara functors, which will play an important role in our arguments.

\begin{lemma}\label{lemma: zero anywhere is zero everywhere}
    If $T$ is a Tambara functor such that $T(X) = 0$ for some nonempty finite $G$-set $X$, then $T$ is the zero Tambara functor (i.e., $T(Y)=0$ for all finite $G$-sets $Y$).
\end{lemma}
\begin{proof}
    Let $Y$ be a finite $G$-set, and let $p\colon X\to G/G$ and $q\colon Y\to G/G$ denote the terminal maps. Since norms and restrictions preserve the unit, we have that $\res_q\nm_p(1_{T(X)}) = 1_{T(Y)}$. Since $p$ is surjective, and restrictions are ring maps, we also have $\res_q\nm_p(0_{T(X)}) = 0_{T(Y)}$.  If $T(X)=0$, we have $0_{T(X)} = 1_{T(X)}$, and thus we also have $0_{T(Y)} = 1_{T(Y)}$ and so $T(Y)$ is the zero ring.
\end{proof}

\subsection{Tambara ideals}

Let $T$ be a $G$-Tambara functor. Throughout this section $X$ will denote a finite $G$-set.
\begin{definition}
    A \emph{Tambara ideal} (or just \emph{ideal}) $I$ of $T$ is a collection of ideals $I(X)\subseteq T(X)$ for all finite $G$-sets $X$ which together are closed under restriction and transfer along all maps of finite $G$-sets.  We also require that $\nm_{f}(I(X))\subseteq I(Y)$ for all \emph{surjective} maps $f\colon X\to Y$ of finite $G$-sets.

    Equivalently, a Tambara ideal is specified by a collection of ideals $I(G/H) \subseteq T(G/H)$ (for all subgroups $H \leq G$) which are closed under the operations $\tr^K_H$, $\nm^K_H$, $\res^K_H$, and $\conj_{g,H}$.
\end{definition}
\begin{remark}
    We explain the seemingly unnatural restriction that Tambara ideals are only closed under norms along surjective maps.  If $\iota\colon X\hookrightarrow X\amalg X$ is the inclusion of the left summand, then through the identification $T(X\amalg X)\cong T(X)^2$, the map $\nm_{\iota}$ is identified with
    \[
        \nm_{\iota}(x) = (x,1)
    \]
    for any Tambara functor $T$.  Since $1$ is never an element of a proper Tambara ideal, it is undesirable to require that ideals be closed under such norms. Indeed, the only Tambara ideal which is closed under arbitrary norms is the ideal $T$. To extend this analogy to ideals in an ordinary commutative ring $R$, we note that while proper ideals are closed under binary multiplication, they are not closed under \emph{nullary} multiplication since $I^{\otimes 0}=R$.
\end{remark}

It is worth stating explicitly that when $I$ is an ideal of a Tambara functor $T$, one can form the quotient $T/I$, defined levelwise by
\[(T/I)(X) = T(X)/I(X),\]
whose structure maps $\tr, \nm, \res, \conj$ are inherited from $T$~\cite[Proposition 2.6]{nakaoka:2012}.

Just like ideals of commutative rings, we can build new ideals out of old ones.  For instance, if $I$ and $J$ are two ideals of $T$ we can take the sum $I+J$ defined by $(I+J)(X) = I(X)+J(X)$.  Just as for commutative rings, one can check that $I+J$ is the smallest ideal of $T$ which contains both $I$ and $J$; see \cite[\S 3.4]{nakaoka:2012} for details.

We can also define the product $IJ$ to be the smallest ideal of $T$ which contains the ideals $I(X)\cdot J(X)\subseteq T(X)$.  That is, $IJ$ is the intersection of all ideals $K$ such that $I(X)\cdot J(X)\subseteq K(X)$ for all finite $G$-sets $X$. An ideal $\mf{p} \subseteq T$ is \emph{prime} if and only if it is proper (i.e., $\mf{p} \neq T$) and $IJ\subseteq \mf{p}$ implies that either $I\subseteq \mf{p}$ or $J\subseteq \mf{p}$.  In practice, this definition is impractical for checking whether or not a given ideal is prime.  We make use of a reformulation, due to Nakaoka.

\begin{definition}\label{definition: multiplicative translates}
    Let $x\in T(G/K)$.  A \emph{multiplicative translate} of $x$ is any element in $T$ of the form $\nm^L_{gHg^{-1}}\conj_{g,H}\res^K_H(x)$ for some $g\in G$ and subgroups $H\leq K$ and $L \geq gHg^{-1}$.  We write $M_x(G/L)$ for the set of multiplicative translates of $x$ in $T(G/L)$, i.e.
    \[M_x(G/L) = \{\nm^L_{gHg^{-1}}\conj_{g,H}\res^K_H(x) : g\in G, H \leq K \cap g^{-1}Lg\}.\]
\end{definition}

\begin{remark}\label{remark: finitely many translates}
    Observe that since $G$ has finitely many elements and subgroups, there are finitely many multiplicative translates of any element $x \in T(G/K)$. In particular, $M_x(G/L)$ is a finite set for all $L \leq G$.
\end{remark}

\begin{remark}
    The definition of $M_x$ can be extended to $x\in T(X)$ for arbitrary $X$ as follows. Picking an isomorphism $X\cong \amalg_{i=1}^n G/K_i$ yields an isomorphism of rings
    \[
        T(X)\cong \prod\limits_{i=1}^n T(G/K_i)
    \]
    and $x\in T(X)$ corresponds to a tuple $(x_1,\dots,x_n)$.  We call the $x_i$ the \emph{coordinates of $x$}.  Note that while the coordinates depend on the choice of isomorphism $X\cong \amalg_{i=1}^n G/K_i$, they are unique up to conjugation and order.  In particular, the union of all $M_{x_i}$ is independent of the choice of isomorphism.  We define $M_x$ to be the union of all $M_{x_i}$, and again $M_x(G/L) = \bigcup_i M_{x_i}(G/L)$ is a finite set for all $L \leq G$.
\end{remark}

The process of forming multiplicative translates is well-behaved under iteration.

\begin{lemma}\label{lem:MT of MT}
    Let $x \in T(G/K)$. Let $x'$ be a multiplicative translate of $x$, and let $x''$ be a multiplicative translate of $x'$. Then $x''$ is a (nonempty) product of multiplicative translates of $x$.
\end{lemma}
\begin{proof}
    Since $\nm$ is multiplicative and commutes with $\conj_g$, it suffices to show that any restriction of $x'$ is a product of multiplicative translates of $x$. This follows from the double-coset formula (see \Cref{example: double coset}) relating norms and restrictions.
\end{proof}

\begin{definition}
    A \emph{generalized product} of two elements $x\in T(X)$ and $y \in T(Y)$, at level $G/H$, is any element of the form $ab$ for $a\in M_x(G/H)$ and $b\in M_y(G/H)$.
\end{definition}


\begin{definition}
    Let $I\subseteq T$ be an ideal and let $x\in T(X)$ and $y\in T(Y)$ be two elements in $T$.  The proposition $\Q(x,y,I)$ is the statement: every generalized product of $x$ and $y$ is in $I$.
\end{definition}

The statement $\Q(x,y,I)$ is the correct generalization to Tambara functors of the statement ``$xy\in I$'' in commutative algebra, and allows for a useful reformulation of prime ideals.  We stress that it is entirely possible, and often the case, that some generalized products of $x$ and $y$ are in $I$ without $\Q(x,y,I)$ being true.

\begin{theorem}[Nakaoka {\cite[Corollary 4.5]{nakaoka:2012}}]\label{Nak characterization of prime}
    A proper Tambara ideal $I \subset T$ is prime if and only if for all $x\in T(G/H)$ and $y\in T(G/K)$ the statement $\Q(x,y,I)$ implies that $x \in I(G/H)$ or $y \in I(G/K)$.
\end{theorem}

\begin{remark}
    The statement of this theorem is not how it is phrased in \cite{nakaoka:2012}.  The formulation of Nakaoka's theorem in terms of the statements $\Q(-,-,-)$ is first given in \cite[Definition 2.6]{CalleGinnett:23}.
\end{remark}

\subsection{Nakaoka spectra}

With the definition of prime ideals in hand, we turn to the generalization of the Zariski spectrum.
\begin{definition}
    The \emph{Nakaoka spectrum} of a Tambara functor $T$, denoted $\Spec(T)$, is the set of prime Tambara ideals of $T$.
\end{definition}

The Nakaoka spectrum is naturally a poset, ordered by inclusion of prime ideals.  It is also a topological space, with the closed sets being those of the form
\[
    V(I) = \{ \mf{p}\in \Spec(T)\mid I\subseteq \mf{p}\}
\]
for $I$ any ideal of $T$. If $f$ is an element of $T(X)$ for some finite $G$-set $X$, we write $V(f)$ for $V(\langle f\rangle)$. This space satisfies several properties familiar from commutative algebra.

\begin{theorem}[{\cite[Theorem A]{CMQSV:24}}]
    The space $\Spec(T)$ is a quasi-compact, sober space.
\end{theorem}

We will have need of the following basis of open subsets.
\begin{definition}\label{defn: D(f)}
    For any $f\in T$, let $D(f) =\{\mf{p}\in \Spec(T)\mid f\notin \mf{p}\}$.
\end{definition}

The set $D(f)\subseteq \Spec(T)$ is open because it is the complement of the closed set $V(\langle f\rangle)$, where $\langle f\rangle$ denotes the smallest ideal of $T$ which contains $f$.

\begin{lemma}
    The collection of $D(f)$ is a basis for the topology on $\Spec(T)$.
\end{lemma}
\begin{proof}
    Every open set in $\Spec(T)$ has the form
    \[
        D(I) := \Spec(T)\setminus V(I) = \{\mf{p}\in\Spec(T)\mid I\not\subseteq\mf{p}\}
    \]
    for some ideal $I$ of $T$.  We have $D(I) = \bigcup_{f\in I} D(f)$, hence every open set is the union of sets of the form $D(f)$.
\end{proof}

One readily checks that $D(f)\cup D(f') = D(\langle f,f'\rangle)$ for any $f$ and $f'$. We can also describe the intersections.
\begin{lemma}\label{lemma: intersection of Ds}
    For any two $e,f\in T$ we have that
    \[
        D(e)\cap D(f) = \bigcup_{z} D(z)
    \]
    where $z$ runs over all generalized products of $e$ and $f$.
\end{lemma}
\begin{proof}
    For the forward inclusion, suppose that $\mf{p}$ is a prime ideal which contains neither $e$ nor $f$.  Then the proposition $\Q(e,f,\mf{p})$ must be false, so there is some generalized product of $e$ and $f$, call it $z$, such that $z\notin \mf{p}$, hence $\mf{p}\in D(z)$.

    For the reverse inclusion it suffices, by symmetry, to show that $D(z)\subseteq D(e)$ for all $z$.  It suffices to show that if $\mf{p}$ contains $e$ then $\mf{p}$ contains $z$.  But this is clear, since $\mf{p}$ is levelwise an ideal and $z$ is divisible by a multiplicative translate of $e$.
\end{proof}
\begin{remark}\label{remark: finite union}
    Since, by \cref{remark: finitely many translates}, the sets $M_e$ and $M_f$ are both finite, there are only finitely many generalized products of $e$ and $f$.  Thus, the union on the right hand side of \cref{lemma: intersection of Ds} is a finite union.
\end{remark}

\section{Pre-Ideals and Annihilators}\label{sec:pre-ideals}

In this section, we introduce some new notions in the theory of commutative algebra of Tambara functors. First, we define \emph{pre-ideals}, a generalization of Tambara ideals which allow for the streamlining of some proofs to follow. Second, we define \emph{annihilator ideals} in direct analogy with ordinary commutative algebra: given any two Tambara ideals $I$ and $J$ of a fixed Tambara functor $T$, the annihilator $(I : J)$ will be the unique largest Tambara ideal of $T$ such that $(I : J) J \subseteq I$.

\subsection{Pre-ideals of Tambara functors}
Let $G$ be a finite group.

\begin{definition}
    A \emph{pre-ideal} of a $G$-Tambara functor $T$ is a collection of subsets $S(G/H)\subseteq T(G/H)$ for all $H\leq G$ satisfying the conditions
    \begin{enumerate}
        \item $S$ is closed under norms and conjugations,
        \item if $x\in S(G/K)$ and $H\leq K$ then $\res^K_H(x)$ is divisible by an element in $S(G/H)$.
    \end{enumerate}
\end{definition}

\begin{example}
    The empty sets $S(G/H) = \emptyset\subseteq T(G/H)$ form a pre-ideal, as do the sets $S'(G/H) = \{1_H\}\subseteq T(G/H)$.
\end{example}

\begin{example}
    If $x\in T(X)$ then $M_x$, the collection of multiplicative translates of $x$, is a pre-ideal of $T$.
\end{example}

\begin{example}
    Every ideal $I$ of $T$ is also a pre-ideal: explicitly, take an ideal $I$ of $T$ and consider the restricted collection of subsets $\cO(I)$ given by $\cO(I)(G/H) = I(G/H)$. Then $\cO(I)$ is a pre-ideal of $T$, and the assignment $I \mapsto \cO(I)$ is injective.
\end{example}

The following lemma is straightforward to verify and we omit the proof.
\begin{lemma}
    If $S_1$ and $S_2$ are pre-ideals of a Tambara functor $T$ then so are $(S_1\cup S_2)(G/H) = S_1(G/H)\cup S_2(G/H)$ and $(S_1S_2)(G/H) = S_1(G/H)S_2(G/H)$.
\end{lemma}

The following lemma characterizes when a pre-ideal is an ideal.
\begin{lemma}\label{lemma: when is preideal an ideal}
    Let $S$ be a non-empty pre-ideal of a Tambara functor $T$.  Then $S = \cO(I)$ for an ideal $I$ if and only if $S$ is closed under levelwise addition and for all $a \in S$ we have $\cO(\langle a \rangle) \subseteq S$.
\end{lemma}
\begin{proof}
    Certainly if $S = \cO(I)$ for some ideal $I$ then $S(G/H) = I(G/H)$ is closed under addition for all $H \leq G$, and $\cO(\langle a \rangle)(G/H) = \langle a \rangle(G/H) \subseteq I(G/H) = S(G/H)$ for all $a \in S, H \leq G$.

    Conversely, supposing that $S$ is a pre-ideal with these properties we will show that $S$ underlies an ideal. We have
    \[
        S \subseteq \bigcup_{a \in S} \cO(\langle a \rangle) \subseteq \bigcup_{a \in S} S = S,
    \]
    so $S = \bigcup_{a \in S} \cO(\langle a \rangle)$. Since each $\cO(\langle a \rangle)$ is closed under the operations $\tr, \nm, \res, \conj$, we conclude that $S$ is closed under these operations. Additionally, $S$ is also closed under addition by assumption. Finally,
    \[
        \cO(T) S = \cO(T) \left(\bigcup_{a \in S} \cO(\langle a \rangle)\right) = \bigcup_{a \in S} \cO(T) \cO(\langle a \rangle) \subseteq \bigcup_{a \in S} \cO(\langle a \rangle) \subseteq S
    \]
    shows that $S$ is a levelwise collection of ideals.
\end{proof}
\begin{remark}\label{remark: when ideal}
    While the previous lemma is stated with $S$ as a pre-ideal, note that it is true, and the same proof applies, whenever $S(G/H)\subseteq T(G/H)$ is a collection of non-empty subsets with no further conditions.
\end{remark}

For any pre-ideal $S$, let $\I(S)$ denote the intersection of ideals of $T$ which contain $S$.  For instance, one immediately has that $\I(\cO(I))=I$.

\begin{example}
    For any $x\in T(X)$ we have that $\I(M_x) = \langle x\rangle$ is the smallest ideal containing $x$. Note that if $x\in T(X)$ has coordinates $x_i\in T(G/H_i)$ for $i=1,\dots, n$ then $\langle x \rangle = \I(M_x)$ is also equal to $\langle x_1,\dots,x_n\rangle$, the smallest ideal containing all the $x_i$.
\end{example}

\begin{proposition}
    Let $S$ be a pre-ideal of $T$.  Then $\I(S)(G/K)$ consists of all elements which can be written as sums of elements of the form $\tr^K_H(xy)$ where $y\in S(G/H)$ and $x\in T(G/H)$ is arbitrary.
\end{proposition}
\begin{proof}[Proof sketch]
    Let $\mathcal{J}(G/K)$ be the collection of all finite sums of elements of the form $\tr^K_H(xy)$ where $y\in S(G/H)$ and $x\in T(G/H)$ is arbitrary.  Certainly every element in $\mathcal{J}(G/K)$ must be contained in every ideal which contains $S$, and so $\mathcal{J}\subseteq \I(S)$.  Conversely, one can check (as Nakaoka does in \cite[Proposition 3.4]{nakaoka:2012}) that the $\mathcal{J}(G/K)$ determine an ideal.  By definition $S\subseteq \mathcal{J}$, and so $\I(S)\subseteq \mathcal{J}$.
\end{proof}

\begin{proposition}\label{prop: I of product}
    If $S_1$ and $S_2$ are two pre-ideals then $\I(S_1S_2) = \I(S_1)\I(S_2)$.
\end{proposition}
\begin{proof}
    Observe that for any $H\leq G$ we have
    \[
        S_1(G/H)S_2(G/H)\subseteq [\I(S_1)(G/H)][\I(S_2)(G/H)]\subseteq [\I(S_1)\I(S_2)](G/H),
    \]
    and so $\I(S_1S_2) \subseteq \I(S_1)\I(S_2)$ by the definition.  To prove the reverse inclusion, suppose that $x\in \I(S_1)(G/K)$ and $y\in \I(S_2)(G/K)$.  Then by the previous proposition we have that $x$ is a sum of elements of the form $\tr^K_{L}(ab)$ where $b\in S_1(G/L)$ and $y$ is the sum of elements of the form $\tr^K_{H}(cd)$ where $d\in S_2(G/H)$.  Taking the product of $x$ and $y$ yields a sum of terms of the form
    \begin{align*}
        \tr^K_{L}(ab)\tr^K_{H}(cd) & = \tr^K_{L}(ab\res^K_L\tr^K_H(cd)) = \tr^K_L\left(ab\cdot \sum\limits_{\gamma\in L\backslash K/H} \tr^L_{L\cap \gamma H\gamma^{-1}}\res^{\gamma H\gamma^{-1}}_{L\cap \gamma H\gamma^{-1}}(\conj_{\gamma}(cd))\right) \\
                                   & = \sum\limits_{\gamma\in L\backslash K/H} \tr^K_{L\cap \gamma H \gamma^{-1}}\left( \res^L_{L\cap \gamma H\gamma^{-1}}(ab)\res^{\gamma H\gamma^{-1}}_{L\cap \gamma H\gamma^{-1}}(\conj_{\gamma}(cd))\right)
    \end{align*}
    where the first and third equalities are Frobenius reciprocity and the second is the double coset formula. Since $S_1$ and $S_2$ are pre-ideals, the terms $\res^{\gamma H\gamma^{-1}}_{L\cap \gamma H\gamma^{-1}}(\conj_{\gamma}(cd))$ and $\res^L_{L\cap \gamma H\gamma^{-1}}(ab)$ are each divisible by elements in $S_2$ and $S_1$, respectively.  Thus, $\tr^K_{L}(ab)\tr^K_{H}(cd)$ is the sum of elements which are in $\I(S_1S_2)$, and we conclude that $xy\in \I(S_1S_2)$.  Since $\I(S_1)\I(S_2)$ is, by definition, the smallest ideal containing every element of the form $xy$ we see that $\I(S_1)\I(S_2)\subseteq \I(S_1S_2)$.
\end{proof}

\begin{corollary}\label{corollary: pre ideals and containment}
    Let $S_1$ and $S_2$ be two pre-ideals and let $J$ be an ideal.  We have $S_1S_2\subseteq J$ if and only if $\I(S_1)\I(S_2)\subseteq J$.
\end{corollary}
\begin{proof}
    We have $S_1S_2\subseteq J$ if and only if $\I(S_1S_2) \subseteq J$, by the definition of $\I(S_1S_2)$,  but $\I(S_1S_2) = \I(S_1)\I(S_2)$, by \cref{prop: I of product}.
\end{proof}

\begin{corollary}
    The statement $\Q(x,y,I)$ is true if and only if $\langle x\rangle\langle y\rangle \subseteq I$.
\end{corollary}
\begin{proof}
    By definition $\Q(x,y,I)$ is the statement that there is a containment $M_xM_y\subseteq I$.  By \cref{corollary: pre ideals and containment}, this is equivalent to $\langle x\rangle\langle y\rangle\subseteq I$.
\end{proof}

\begin{corollary}\label{cor: Q and translates}
    Suppose that $\Q(x,y,I)$ is true, $x'\in \langle x\rangle$, and $y'\in \langle y\rangle$.  Then  $\Q(x',y',I)$ is true.
\end{corollary}
\begin{proof}
    We have $\langle x'\rangle \subseteq \langle x\rangle$ and $\langle y'\rangle \subseteq \langle y\rangle$ and so if $\langle x\rangle\langle y\rangle \subseteq I$ then $\langle x'\rangle\langle y'\rangle \subseteq \langle x\rangle\langle y\rangle \subseteq I$.
\end{proof}

\cref{cor: Q and translates} allows us to prove the following technical result, necessary for the proof of \cref{theorem: main theorem 1 intro}, which allows us to construct examples of prime ideals in Tambara functors.

\begin{proposition}\label{prop: prime ideal machine}
    Let $T$ be a Tambara functor. If $S \subseteq T(G/H)$ is a multiplicatively closed subset which does not contain $0$ then there is at least one ideal $\mf{p}$ of $T$ which is maximal among ideals disjoint from $S$. Moreover, if $S$ is nonempty, then $\mf{p}$ is a prime ideal.
\end{proposition}
\begin{proof}
    The collection of ideals disjoint from $S$ is nonempty, as it contains the zero ideal. By Zorn's lemma, there is an ideal $\mf{p}$ which is maximal with respect to the property that $\mf{p}(G/H)\cap S=\varnothing$ (note that the union of a chain of Tambara ideals is a Tambara ideal).

    Now assuming $S \neq \varnothing$, it remains to be shown that $\mf{p}$ is prime. Since $S$ is nonempty, we immediately find that $\mf{p}$ is a proper ideal of $T$. Next, toward a contradiction, suppose there exist $y\in T(G/K)$ and $z\in T(G/L)$ such that $\Q(y,z,\mf{p})$ is true, but neither $y$ nor $z$ is in $\mf{p}$. Then $\mf{p} + \langle y \rangle$ and $\mf{p} + \langle z \rangle$ are ideals strictly larger than $\mf{p}$, so they intersect $S$. Thus, there exist $p_1,p_2\in \mf{p}(G/H)$ and $y',z'\in T(G/H)$, with $y'\in \langle y\rangle$ and $z'\in \langle z\rangle$ respectively, such that
    \begin{align*}
        s_1 & = p_1+y' \\
        s_2 & = p_2+z'
    \end{align*}
    for some $s_1,s_2\in S$.  Multiplying the two equations we have
    \[
        s_1s_2= p_3+y'z'
    \]
    where $p_3\in \mf{p}(G/H)$.  Now, by \cref{cor: Q and translates}, we have that $\Q(y,z,\mf{p})$ implies that $\Q(y',z',\mf{p})$ is true and hence $y'z'\in \mf{p}$. But then $s_1s_2\in \mf{p}(G/H)\cap S$, which is a contradiction.
\end{proof}

\subsection{Annihilators}

In this brief subsection we observe that the theory of annihilators admits a straightforward generalization to ideals of Tambara functors. If $I$ and $J$ are two ideals of a commutative ring $R$, then one can define
\[
    (I:J) = \{a\in R\mid aJ \subseteq I\}
\]
and it is straightforward to check that this is an ideal. When $I=0$, this is precisely the annihilator of $J$.

\begin{definition}
    Let $I$ and $J$ be ideals of a Tambara functor $T$. Define
    \[
        (I: J)(X) = \{a\in T(X)\mid M_a J\subseteq I\}.
    \]
\end{definition}
\begin{lemma}
    If $a\in (I:J)$ then for any $x\in J$ we have $\Q(a,x,I)$ is true.
\end{lemma}
\begin{proof}
    The statement $\Q(a,x,I)$ is the same as $M_aM_x\subseteq I$.  Since $J$ is an ideal and $x\in J$ we have $M_x\subseteq J$ and so
    \[
        M_aM_x\subseteq M_a J\subseteq I
    \]
    where the second inclusion uses the fact that $a\in (I:J)$.
\end{proof}
\begin{lemma}
    There is an equality
    \[
        (I: J)(X) =  \{a\in T(X)\mid \langle a\rangle J\subseteq I\}.
    \]
    Moreover, $(I : J)$ is a Tambara ideal.
\end{lemma}
\begin{proof}
    The first claim follows immediately from \cref{corollary: pre ideals and containment}, since $\I(M_a)J = \langle a\rangle J$. For the second claim, we apply \cref{lemma: when is preideal an ideal,remark: when ideal}.  It is immediate that $(I:J)$ is closed under addition.  If $a\in (I:J)$ and $b\in \langle a\rangle$ we have $\langle b\rangle J\subseteq \langle a\rangle J\subseteq I$ and so $b\in (I:J)$.
\end{proof}

\begin{proposition}
    The ideal $(I:J)$ is the largest ideal $L$ such that $LJ\subseteq I$.
\end{proposition}
\begin{proof}
    Observe that $I\subseteq (I:J)$, and that the ideal $(I:J)/I\subseteq T/I$ is the ideal $(0:(J+I)/I)\subseteq T/I$.  From this, we see it suffices to consider the case $I=0$.  In this case, the claim is that $(0:J)$ is the largest ideal of $T$ so that $(0:J)J=0$.  Indeed, if $L$ is any other ideal so that $LJ=0$ then every element $x\in L$ must satisfy $M_xJ =0$ and so $L\subseteq (0:J)$.
\end{proof}

\section{Nilpotence in Tambara functors}\label{sec: nilpotents}

Let $R$ be a commutative ring, and let $x \in R$. It is a key feature of commutative algebra that the following are equivalent:
\begin{enumerate}
    \item $x$ is nilpotent, i.e., $x^n = 0$ for some natural number $n$;
    \item $x$ lies in every prime ideal of $R$;
    \item $R[1/x] = 0$.
\end{enumerate}

In equivariant commutative algebra, one would hope these conditions remain equivalent, thus furnishing a definition of ``nilpotent'' in this setting. Unfortunately this is not the case -- condition (3) is generally weaker.

For example, consider the element \[t - p = [C_p/e] - p [C_p/C_p]\] in $\uA(C_p/C_p)\cong \mathbb{Z}[t]/(t^2-pt)$, where $\uA$ is the Burnside $C_p$-Tambara functor for $p$ a prime. Since $0$ is a prime ideal of $\uA$, $t-p$ does not lie in every prime ideal of $\uA$. On the other hand, $\res^{C_p}_e(t-p) = 0$, so $\uA[1/(t-p)] = 0$; the definition of $\uA[1/(t-p)]$, and more general localizations of Tambara functors, is recalled in \cref{sec: kilradical}.

In the remainder of this section we show that conditions (1) and (2) are equivalent and use this to study radicals of ideals, in general.  We will say more about condition (3) in \cref{sec: kilradical}.

\begin{definition}
    For a Tambara functor $T$ we write $\nil(T)$ for the intersection of all prime Tambara ideals.
\end{definition}

\begin{definition}
    The \emph{levelwise nilradical} of a $G$-Tambara functor is the ideal $N(T)$ such that $N(T)(X)$ is the nilradical of $T(X)$ for any finite $G$-set $X$.
\end{definition}

It is not, a priori, obvious that $N(T)$ is actually an ideal.  We show in \cref{thm: main} below that $N(T) = \nil(T)$, and thus $N(T)$ is an ideal.

\begin{lemma}\label{lemma: N is in nil}
    There is an inclusion $N(T)\subseteq \nil(T)$.
\end{lemma}
\begin{proof}
    This is immediate from \cite[Theorem 4.7]{CMQSV:24} which says that every Tambara prime ideal is levelwise radical.
\end{proof}

We now show that $N(T) = \nil(T)$.  The technical heart of the proof is really \cref{prop: prime ideal machine}, proven in the last section.

\begin{theorem}\label{thm: main}
    For any $G$-Tambara functor $T$ we have $\nil(T)=N(T)$.
\end{theorem}
\begin{proof}
    By \cref{lemma: N is in nil}, we have $N(T)\subseteq \nil(T)$, so it suffices to prove the reverse inclusion. Suppose that $x\in T(X)$ is any element which is not in $N(T)(X)$. Then there is some coordinate $x_i\in T(G/H)$ of $x$ which is not in $N(T)(G/H)$. We will construct a prime ideal of $T$ that does not contain $x_i$, which implies $x_i\notin \nil(T)$, and hence $x\notin \nil(T)$.  Thus $N(T)^c\subseteq \nil(T)^c$ and the proof is complete by taking complements.

    Since $x_i$ is not nilpotent, the set $S = \{x^n_i\}_{n=1}^{\infty}$ is a multiplicatively closed subset of $T(G/H)$ which does not contain $0$.  By \cref{prop: prime ideal machine} there is at least one prime ideal $\mf{p}$ which is disjoint from $S$.  In particular, $x_i\notin \mf{p}(G/H)$ and the proof is complete.
\end{proof}

In \cite[Proposition 4.16]{nakaoka:2012}, Nakaoka defines the radical of an ideal in a Tambara functor $T$ and proves some desirable properties of this construction.
\begin{definition}[Nakaoka]
    If $I \subseteq T$ is an ideal then the \emph{radical} of $I$, denoted $\sqrt{I}$, is the ideal defined by
    \[\sqrt{I}(X) = \{x \in T(X) : \exists n \in \mathbb{N} \text{ such that } \langle x \rangle^n \subseteq I\}\]
    for $X$ a finite $G$-set.
\end{definition}

Nakaoka showed that $\sqrt{I}$ is contained in the intersection of all prime ideals containing $I$, but left open the question of whether or not this containment could be proper. We end this section by proving that it is always an equality. The key observation is that the radical of the zero ideal in $T$ is $\nil(T)$.

\begin{theorem}\label{thm: nakaoka radical}
    Let $T$ be a $G$-Tambara functor and let $y\in T(Y)$.  The following are equivalent:
    \begin{enumerate}
        \item $y$ is nilpotent;
        \item $\langle y\rangle^n=0$ for some $n$.
    \end{enumerate}
\end{theorem}
\begin{proof}
    To see that $(2)$ implies $(1)$, simply observe that $y^n\in \langle y\rangle^n$.  We focus, then, on the implication $(1)\implies (2)$. Suppose that $y^m=0$ for some $m\geq1$. We want to find a sufficiently large $n$ so that $\langle y\rangle^n=0$.  Since $\langle y\rangle^n = \I(M_y)^n$, \cref{corollary: pre ideals and containment} implies that it suffices to find a large $n$ so that $(M_y)^n=0$.  That is, we need an $N>0$ large enough so that for all $n>N$ we have that any product $\overline{y}_1\cdots\overline{y}_n \in T(G/K)$, where each $\overline{y}_i\in M_y$, is zero. Since $y^m = 0$, we have $z^m = 0$ for any $z\in M_y$. Since $M_y$ is a finite set (\cref{remark: finitely many translates}), the pigeonhole principle then implies any sufficiently large product of elements in $M_y$ must be divisible by $z^m$ for some $z$, hence is zero.\qedhere
\end{proof}

\begin{corollary}\label{cor: nakaoka radicals}
    Let $I$ be any ideal of a Tambara functor $T$. Then $\sqrt{I}(X)$ is the radical of $I(X)$. Moreover, $\sqrt{I}$ is the intersection of all prime ideals which contain $I$.
\end{corollary}
\begin{proof}
    The case $I=0$ is immediate from  \cref{thm: nakaoka radical}, \cref{thm: main}, and the definition of $\sqrt{0}$. The general case comes from considering the zero ideal in the Tambara functor $T/I$.
\end{proof}

We also find, as in commutative algebra, that ideals finitely generated by nilpotent elements are nilpotent.
\begin{corollary}
    For $i=1,\dots,n$, let $x_i\in T(G/H_i)$ be elements in a Tambara functor and let $I = \langle x_1,\dots,x_n\rangle$ be the ideal generated by the $x_i$.  If each $x_i$ is nilpotent, then $I$ is a nilpotent ideal, in the sense that $I^N=0$ for some $N$.
\end{corollary}
\begin{proof}
    Observe that we have $I =  \sum_{i=1}^n \langle x_i\rangle$. Pick a $k>0$ large enough so that $\langle x_i\rangle^{k}=0$ for all $i$; this exists by \cref{thm: nakaoka radical}.  Using the fact that multiplication of ideals distributes over addition of ideals (\cite[Proposition 3.15]{nakaoka:2012}) we have that $\left(\sum_{i=1}^n \langle x_i\rangle\right)^{nk}=0$, hence $I^{nk}=0$. Alternatively, the claim follows from noting that the element $x\in T(\amalg G/H_i)$ with coordinates $x_i$ is nilpotent, and observing that $I = \langle x\rangle$.
\end{proof}

\section{Spectral Spaces}\label{sec: spectral spaces}

Recall that a topological space is said to be \emph{spectral} if it is homeomorphic to the Zariski spectrum of a commutative ring.  An equivalent condition is that the space is
\begin{enumerate}
    \item sober (every nonempty irreducible closed subset has a unique generic point), and
    \item the set of quasi-compact opens is closed under finite intersection and forms a basis.
\end{enumerate}

We established that $\Spec(T)$ is sober in \cite[Theorem A]{CMQSV:24}. Note that condition (2) includes also the empty intersection, i.e., a spectral space must be quasi-compact  -- this is proven of $\Spec(T)$ in \textit{loc.\ cit.}  To show that $\Spec(T)$ is spectral, it remains to be shown that the set of quasi-compact opens forms a basis which is closed under binary intersection.

We claim it is enough to prove that each $D(f)$, as defined in \cref{defn: D(f)}, is quasi-compact, since the $D(f)$'s form a basis for the topology of $\Spec(T)$. Indeed, if this is established, then every quasi-compact open set is a finite union of $D(f)$'s, and the fact that $D(f) \cap D(g)$ is a finite union of $D(h)$'s (\cref{lemma: intersection of Ds,remark: finite union}) implies that the intersection of any two quasi-compact opens will remain quasi-compact.

\begin{proposition}
    For any Tambara functor $T$ and any element $f$ in $T$ the basic open $D(f)$ is quasi-compact.
\end{proposition}
\begin{proof}
    Write $D(f) = \bigcup_{i \in I} D(g_i)$ for some collection of elements $\{g_i : i \in I\}$. In particular,
    \[\forall \mf{p} \in \Spec(T) \, (\{g_i : i \in I\} \subseteq \mf{p} \implies f \in \mf{p}).\]
    Equivalently, (the image of) $f$ lies in $\nil(T/\langle g_i : i \in I \rangle)$. By \Cref{thm: main}, this means that the image of $f$ is nilpotent in $T/\langle g_i : i \in I \rangle$, i.e., $f^n \in \langle g_i : i \in I \rangle$ for some natural number $n$. Then $f^n \in \langle g_{i_1}, \dots, g_{i_k} \rangle$ for some $i_1, \dots, i_k \in I$. Now the image of $f$ lies in $\nil(T/\langle g_{i_1}, \dots, g_{i_k} \rangle)$, so
    \[\forall \mf{p} \in \Spec(T) \, (\{ g_{i_1}, \dots, g_{i_k} \} \subseteq \mf{p} \implies f \in \mf{p}).\]
    Equivalently, $V(\langle g_{i_1}, \dots, g_{i_k} \rangle) \subseteq V(f)$, i.e., $D(f) \subseteq \bigcup_{j=1}^k D(g_{i_j})$.
\end{proof}
\begin{corollary}\label{cor: spectral}
    The Nakaoka spectrum of any Tambara functor is a spectral space.
\end{corollary}

\section{The Kilradical}\label{sec: kilradical}

In this final section we study the collection of elements in a Tambara functor  $T$ which kill $T$ upon being inverted. In commutative rings, this is exactly the collection of nilpotent elements.  For Tambara functors, this collection can be strictly larger than the set of nilpotents, and we call these elements \emph{kilpotents}.  To give a precise definition, we recall the existence of localizations for Tambara functors.
\begin{proposition}\label{proposition: localizations exist}
    Let $T$ be a $G$-Tambara functor and let $x\in T$.  Then there is a Tambara functor $T[1/x]$ and a map $T\to T[1/x]$ which is initial among all maps of Tambara functors $f\colon T\to S$ such that $f(x)$ is a unit.
\end{proposition}
\begin{proof}
    This is a special case of \cite[Proposition 4.6]{Nak:fractions}.  Specifically, in Nakaoka's notation, it is the case where $\mathscr{S}$ is the multiplicative sub-semi-Mackey functor generated by $x$, and $\mathscr{S}'$ is the multiplicative sub-semi-Mackey functor of units.
\end{proof}

\begin{definition}
    An element $x\in T(X)$ is a \emph{kilpotent} if $T[1/x]=0$. The collection of all kilpotent elements in $T$ is called the \emph{kilradical}, denoted $\kil(T)$.
\end{definition}

As first examples, it is easy to check that every nilpotent element in a Tambara functor is a kilpotent.  Moreover, every element $x\in T(G/H)$ such that $\res^H_e(x)$ is $0$ (or nilpotent) is also a kilpotent, as inverting $x$ necessarily inverts $\res^H_e(x)$.

\begin{example}\label{example: bad kilpotents}
    We give an example of an $x\in \kil(T)$ which is neither nilpotent  nor in the kernel of a restriction map. Let $G=C_2$, and consider the Tambara functor
    \[\begin{tikzcd}
            {\mathbb{R}} \\
            {\mathbb{R}\times \mathbb{R}}
            \arrow[from=1-1, to=2-1]
            \arrow["{+}", shift left=3, from=2-1, to=1-1]
            \arrow["\times"', shift right=3, from=2-1, to=1-1]
            \arrow["{\mathrm{swap}}"', from=2-1, to=2-1, loop, in=305, out=235, distance=10mm]
        \end{tikzcd}\]
    where the restriction map is the diagonal.  The element $(0,1)$ at the $C_2/e$-level is not nilpotent (and not in the kernel of a restriction), but is in $\kil(T)$, as $\nm^{C_2}_e(0,1) = 0$.
\end{example}
\begin{remark}
    We warn the reader that $\kil(T)$ is not necessarily an ideal of the Tambara functor $T$.  Indeed, in the last example, $\kil(T)(C_2/e)$ consists of all elements of the form $(r,0)$ or $(0,r)$.  As this collection is not closed under addition, this cannot be an ideal.  We show below that when the Weyl action of $G$ on $T(G/e)$ is trivial then $\kil(T)$ is an ideal.
\end{remark}

Heuristically, an element is a kilpotent if some product of multiplicative translates of $x$ can yield $0$. The next theorem makes this heuristic explicit, and more checkable.

\begin{theorem}\label{thm: kilradical}
    An element $x\in T(G/H)$ is a kilpotent if and only if
    \[
        \prod_{g\in G}\conj_g\res^H_e(x)
    \]
    is nilpotent in $T(G/e)$.
\end{theorem}

\begin{proof}
    For brevity, let us write $y = \prod_{g\in G}\conj_g\res^H_e(x)$.

    Suppose first that $y$ is nilpotent. Let $\varphi \colon T \to T[1/x]$ be the localization map. Since $\varphi_H(x)$ is a unit and restriction and conjugation are ring homomorphisms, $\varphi_e(y)$ is a unit. Moreover, since $y$ is nilpotent, $\varphi_e(y)$ is nilpotent. Thus, $T[1/x](G/e) = 0$, and so by \cref{lemma: zero anywhere is zero everywhere} we have $T[1/x] = 0$, i.e., $x$ is kilpotent.

    In the other direction, suppose $y$ is not nilpotent. Since $y$ is fixed by the Weyl action of $G$ on $T(G/e)$, this Weyl action descends to an action of $G$ on $T(G/e)[1/y]$ (and the localization map $T(G/e) \to T(G/e)[1/y]$ is $G$-equivariant). We then obtain a morphism of Tambara functors
    \[
        T \to \FP(T(G/e)) \to \FP(T(G/e)[1/y])
    \]
    where the first map is the unit; see \Cref{example FP}.  The element $x \in T(G/H)$ is sent to \[\frac{\res^H_e(x)}{1} \in T(G/e)[1/y]^H = \FP(T(G/e)[1/y])(G/H)\] by this composite, which is a unit since it divides the element $\frac{y}{1}$. Thus, we obtain a factorization
    \[T \to T[1/x] \to \FP(T(G/e)[1/y]).\]
    Since $y$ is not nilpotent, we have $T(G/e)[1/y] \neq 0$, whence $\FP(T(G/e)[1/y]) \neq 0$. Thus $T[1/x] \neq 0$, i.e., $x$ is not kilpotent.\qedhere
\end{proof}

\begin{remark}
    Let $X$ be an arbitrary $G$-set, and $x\in T(X)$ with coordinates $x_i\in T(G/H_i)$ for $i=1,\dots,n$.  One can check, using essentially the same argument as above, that $x$ is a kilpotent if and only if
    \[
        \prod_{i=1}^n \prod_{g\in G} c_g\res^{H_i}_e(x_i)\in T(G/e)
    \]
    is a nilpotent.
\end{remark}

\begin{corollary}\label{cor:  trivial action means kilradical is an ideal}
    If $T$ is a Tambara functor with $G$ acting trivially on $T(G/e)$, then $\kil(T)$ is the ideal given levelwise by  $\kil(T)(G/H) = (\res^H_e)^{-1}(\nil(T(G/e)))$.
\end{corollary}
\begin{proof}
    We have that $x\in T(G/H)$ is kilpotent if and only if
    \[\prod_{g \in G} \conj_g \res^H_e(x) = \res^H_e(x)^{\lvert G \rvert}\]
    is nilpotent, if and only if $\res^H_e(x)$ is nilpotent.
\end{proof}
\begin{remark}
    We stress that for an arbitrary $G$-set $X$ the ideal $\kil(T)(X)$ consists of all $x\in T(X)$ where \emph{every} coordinate is a kilpotent; this is generally smaller than the collection of all kilpotents in $T(X)$.
\end{remark}

\begin{corollary}\label{cor: kil of A}
    $\kil(\uA)(G/H) = \ker \res^H_e$.
\end{corollary}
\begin{proof}
    $\uA(G/e) \cong \mathbb{Z}$ has trivial Weyl action and trivial nilradical.
\end{proof}

This illustrates the contrast between the nilradical and kilradical: we have $\kil(\uA)(G/H) \neq 0$ whenever $H \neq e$, since (for example) $\res^H_e([H/e] - \lvert H \rvert) = 0$. On the other hand, one always has $\nil(\uA) = 0$, since $0$ is a prime ideal of $\uA$~\cite[Theorem 4.40]{nakaoka:2012}.

We end the paper with a discussion of kilpotents in equivariant stable homotopy theory.  The (non-equivariant) stable homotopy groups of spheres $\pi_\bullet(\mathbb{S})$ (where $\bullet \in \ZZ$) form a graded-commutative ring, and Nishida's nilpotence theorem says that every element in nonzero degree is nilpotent \cite{Nishida}. These are also known as the non-equivariant \emph{stable stems}. The $G$-equivariant stable stems, which we denote by $\pi_{\star}^s$ (where $\star \in \mathrm{RO}(G)$), form an $\mathrm{RO}(G)$-graded Tambara functor; see \cite{AngeltveitBohmann} for definitions. The important feature we make use of here is that there is an isomorphism of $\mathrm{RO}(G)$-graded rings
\[\pi_\star^s(G/e) \cong \pi_{\lvert \star \rvert}(\mathbb{S})\]
where $\lvert \star \rvert$ denotes the virtual dimension of an element $\star \in \mathrm{RO}(G)$, i.e., $\lvert [V] - [W] \rvert = \dim V - \dim W$.

In contrast to the non-equivariant setting, there can be non-nilpotent elements in $\pi^s_{\alpha}(G/H)$ with $\alpha \neq 0$. We note that a theorem of Iriye says that the nilpotent elements are precisely those which are torsion \cite{Iriye}. A study of the non-torsion (and therefore non-nilpotent) elements in $\pi^s_{\star}$ is carried out in \cite{GreenleesQuigley}.  In the following example, we describe the kilpotents in $\pi_{\star}^s$.

\begin{example}\label{stable stems}
    A (possibly non-homogeneous) element $x$ in $\pi_{\star}^s(G/H)$ is kilpotent if and only if $\prod_{g \in G} g \cdot \res^H_e(x)$ is nilpotent. Since $\pi_\bullet(\mathbb{S})$ is $0$ in negative degrees, it follows that the projection
    \[\pi_\star^s(G/e) \cong \pi_{\lvert \star \rvert}(\mathbb{S}) \to \bigoplus_{\substack{\alpha \in \mathrm{RO}(G) \\ \lvert \alpha \rvert = 0}} \pi_\alpha^s(G/e)\]
    is an $\mathrm{RO}(G)$-graded ring homomorphism (where the right-hand side is viewed as an $\mathrm{RO}(G)$-graded ring by putting $0$ in each missing component). Nishida's nilpotence theorem then implies that $x$ is a kilpotent if and only if it is in the kernel of the composite
    \[\pi^s_\star(G/H) \xrightarrow{\res} \pi^s_\star(G/e) \to \bigoplus_{\substack{\alpha \in \mathrm{RO}(G) \\ \lvert \alpha \rvert = 0}} \pi_\alpha^s(G/e).\]
    Hence, the kilpotents of $\pi_\star^s(G/H)$ form a graded ideal of $\pi_{\star}^s(G/H)$, i.e., an element of $\pi_\star^s(G/H)$ is a kilpotent if and only if all of its homogeneous components are kilpotent. It follows immediately from this observation that kilpotents are closed under transfer and restriction, as these operations are additive, preserve virtual dimension, and preserve the kernel of restrictions to the $G/e$ level.

    To see that kilpotents are closed under norm, the Tambara reciprocity formula for the norm of a sum, together with closure under transfer, implies that it is enough to prove that homogeneous kilpotent elements are closed under norms.  In nonzero virtual dimension this is clear, since norm changes virtual dimension by multiplication by a positive integer.  In virtual dimension zero this follows from the multiplicative double coset formula.
\end{example}

\printbibliography

@article{AngeltveitBohmann,
  author     = {Angeltveit, Vigleik and Bohmann, Anna Marie},
  title      = {Graded {T}ambara functors},
  journal    = {J. Pure Appl. Algebra},
  fjournal   = {Journal of Pure and Applied Algebra},
  volume     = {222},
  year       = {2018},
  number     = {12},
  pages      = {4126--4150},
  issn       = {0022-4049,1873-1376},
  mrclass    = {55P91 (55N91 55P43 55Q91)},
  mrnumber   = {3818296},
  mrreviewer = {Gerd\ Laures},
  doi        = {10.1016/j.jpaa.2018.02.023},
  url        = {https://doi.org/10.1016/j.jpaa.2018.02.023}
}

@article{Brun,
  author     = {Brun, Morten},
  title      = {Witt vectors and {T}ambara functors},
  journal    = {Adv. Math.},
  fjournal   = {Advances in Mathematics},
  volume     = {193},
  year       = {2005},
  number     = {2},
  pages      = {233--256},
  issn       = {0001-8708,1090-2082},
  mrclass    = {13K05},
  mrnumber   = {2136887},
  mrreviewer = {Barry\ H.\ Dayton},
  doi        = {10.1016/j.aim.2004.05.002},
  url        = {https://doi.org/10.1016/j.aim.2004.05.002}
}

@article{CalleGinnett:23,
  author     = {Calle, Maxine and Ginnett, Sam},
  title      = {The spectrum of the {B}urnside {T}ambara functor of a cyclic
                group},
  journal    = {J. Pure Appl. Algebra},
  fjournal   = {Journal of Pure and Applied Algebra},
  volume     = {227},
  year       = {2023},
  number     = {8},
  pages      = {Paper No. 107344, 23},
  issn       = {0022-4049,1873-1376},
  mrclass    = {19A22 (13A15 18B99 55P91)},
  mrnumber   = {4552396},
  mrreviewer = {Geoffrey\ M. L. Powell},
  doi        = {10.1016/j.jpaa.2023.107344},
  url        = {https://doi.org/10.1016/j.jpaa.2023.107344}
}

@article{CCMQSV,
  author   = {Calle, Maxine Elena and Chan, David and Mehrle, David and
              Quigley, J. D. and Spitz, Ben and Van Niel, Danika},
  title    = {The spectrum of the {B}urnside {T}ambara functor},
  journal  = {Int. Math. Res. Not. IMRN},
  fjournal = {International Mathematics Research Notices. IMRN},
  year     = {2026},
  number   = {2},
  pages    = {Paper No. rnaf388, 26},
  issn     = {1073-7928,1687-0247},
  mrclass  = {19A22 (18A22 55P91)},
  mrnumber = {5019790},
  doi      = {10.1093/imrn/rnaf388},
  url      = {https://doi.org/10.1093/imrn/rnaf388}
}

@article{CMQSV:24,
  author   = {Chan, David and Mehrle, David and Quigley, J. D. and Spitz,
              Ben and Van Niel, Danika},
  title    = {On the {T}ambara affine line},
  journal  = {Adv. Math.},
  fjournal = {Advances in Mathematics},
  volume   = {493},
  year     = {2026},
  pages    = {Paper No. 110908, 74},
  issn     = {0001-8708,1090-2082},
  mrclass  = {55P91 (13A50 13B21 19A22)},
  mrnumber = {5043630},
  doi      = {10.1016/j.aim.2026.110908}
}

@article{GreenleesQuigley,
  author     = {Greenlees, J. P. C. and Quigley, J. D.},
  title      = {Ranks of {$RO(G)$}-graded stable homotopy groups of spheres
                for finite groups {$G$}},
  journal    = {Proc. Amer. Math. Soc. Ser. B},
  fjournal   = {Proceedings of the American Mathematical Society. Series B},
  volume     = {10},
  year       = {2023},
  pages      = {101--113},
  issn       = {2330-1511},
  mrclass    = {55P91 (55Q45)},
  mrnumber   = {4569751},
  mrreviewer = {Oliver\ R\"ondigs},
  doi        = {10.1090/bproc/140}
}

@article{HHR16,
  author     = {Hill, M. A. and Hopkins, M. J. and Ravenel, D. C.},
  title      = {On the nonexistence of elements of {K}ervaire invariant one},
  journal    = {Ann. of Math. (2)},
  fjournal   = {Annals of Mathematics. Second Series},
  volume     = {184},
  year       = {2016},
  number     = {1},
  pages      = {1--262},
  issn       = {0003-486X,1939-8980},
  mrclass    = {55P91 (55N22 55P42 55Q45 55T15 55U35 57R15)},
  mrnumber   = {3505179},
  mrreviewer = {Paul\ G.\ Goerss},
  doi        = {10.4007/annals.2016.184.1.1},
  url        = {https://doi.org/10.4007/annals.2016.184.1.1}
}

@article{hill:2017,
  author     = {Hill, Michael A.},
  title      = {On the {A}ndr\'e-{Q}uillen homology of {T}ambara functors},
  journal    = {J. Algebra},
  fjournal   = {Journal of Algebra},
  volume     = {489},
  year       = {2017},
  pages      = {115--137},
  issn       = {0021-8693,1090-266X},
  mrclass    = {18G55 (20J05)},
  mrnumber   = {3686975},
  mrreviewer = {Vigleik\ Angeltveit},
  doi        = {10.1016/j.jalgebra.2017.06.029},
  url        = {https://doi.org/10.1016/j.jalgebra.2017.06.029}
}

@article{Hill:Localization,
  author     = {Hill, Michael A.},
  title      = {Equivariant chromatic localizations and commutativity},
  journal    = {J. Homotopy Relat. Struct.},
  fjournal   = {Journal of Homotopy and Related Structures},
  volume     = {14},
  year       = {2019},
  number     = {3},
  pages      = {647--662},
  issn       = {2193-8407,1512-2891},
  mrclass    = {55P42 (55P60 55P91)},
  mrnumber   = {3987553},
  mrreviewer = {Samik\ Basu},
  doi        = {10.1007/s40062-018-0226-2},
  url        = {https://doi.org/10.1007/s40062-018-0226-2}
}

@incollection{HillHopkins:closure,
  author     = {Hill, M. A. and Hopkins, M. J.},
  title      = {Equivariant multiplicative closure},
  booktitle  = {Algebraic topology: applications and new directions},
  series     = {Contemp. Math.},
  volume     = {620},
  pages      = {183--199},
  publisher  = {Amer. Math. Soc., Providence, RI},
  year       = {2014},
  isbn       = {978-0-8218-9474-3},
  mrclass    = {55P43 (55P60 55P91)},
  mrnumber   = {3290092},
  mrreviewer = {Steven\ R.\ Costenoble},
  doi        = {10.1090/conm/620/12372},
  url        = {https://doi.org/10.1090/conm/620/12372}
}

@article{HillMehrleQuigley,
  author     = {Hill, Michael A. and Mehrle, David and Quigley, James D.},
  title      = {Free incomplete {T}ambara functors are almost never flat},
  journal    = {Int. Math. Res. Not. IMRN},
  fjournal   = {International Mathematics Research Notices. IMRN},
  year       = {2023},
  number     = {5},
  pages      = {4225--4291},
  issn       = {1073-7928,1687-0247},
  mrclass    = {55N25 (16S10 18Gxx)},
  mrnumber   = {4565666},
  mrreviewer = {David\ Barnes},
  doi        = {10.1093/imrn/rnab361},
  url        = {https://doi.org/10.1093/imrn/rnab361}
}

@article{Iriye,
  author   = {Iriye, Kouyemon},
  title    = {The nilpotency of elements of the equivariant stable homotopy
              groups of spheres},
  journal  = {J. Math. Kyoto Univ.},
  fjournal = {Journal of Mathematics of Kyoto University},
  volume   = {22},
  year     = {1982},
  number   = {2},
  pages    = {257--259},
  issn     = {0023-608X},
  mrclass  = {55Q40},
  mrnumber = {666975},
  doi      = {10.1215/kjm/1250521814},
  url      = {https://doi.org/10.1215/kjm/1250521814}
}

@article{LRZ,
  author   = {Lindenstrauss, Ayelet and Richter, Birgit and Zou, Foling},
  title    = {Loday constructions of {T}ambara functors},
  journal  = {J. Algebra},
  fjournal = {Journal of Algebra},
  volume   = {683},
  year     = {2025},
  pages    = {278--306},
  issn     = {0021-8693,1090-266X},
  mrclass  = {55P91 (13D03)},
  mrnumber = {4929908},
  doi      = {10.1016/j.jalgebra.2025.06.016},
  url      = {https://doi.org/10.1016/j.jalgebra.2025.06.016}
}

@article{Nak:fractions,
  author     = {Nakaoka, Hiroyuki},
  title      = {On the fractions of semi-{M}ackey and {T}ambara functors},
  journal    = {J. Algebra},
  fjournal   = {Journal of Algebra},
  volume     = {352},
  year       = {2012},
  pages      = {79--103},
  issn       = {0021-8693,1090-266X},
  mrclass    = {19A22 (20M99)},
  mrnumber   = {2862176},
  mrreviewer = {Robert\ Hartmann},
  doi        = {10.1016/j.jalgebra.2011.11.013},
  url        = {https://doi.org/10.1016/j.jalgebra.2011.11.013}
}

@article{nakaoka:2012,
  author     = {Nakaoka, Hiroyuki},
  title      = {Ideals of {T}ambara functors},
  journal    = {Adv. Math.},
  fjournal   = {Advances in Mathematics},
  volume     = {230},
  year       = {2012},
  number     = {4-6},
  pages      = {2295--2331},
  issn       = {0001-8708,1090-2082},
  mrclass    = {19A22},
  mrnumber   = {2927371},
  mrreviewer = {Robert\ Hartmann},
  doi        = {10.1016/j.aim.2012.04.021},
  url        = {https://doi.org/10.1016/j.aim.2012.04.021}
}

@article{nakaoka:2014,
  author   = {Nakaoka, Hiroyuki},
  title    = {The spectrum of the {B}urnside {T}ambara functor on a finite
              cyclic {$p$}-group},
  journal  = {J. Algebra},
  fjournal = {Journal of Algebra},
  volume   = {398},
  year     = {2014},
  pages    = {21--54},
  issn     = {0021-8693,1090-266X},
  mrclass  = {20Dxx},
  mrnumber = {3123752},
  doi      = {10.1016/j.jalgebra.2013.09.010},
  url      = {https://doi.org/10.1016/j.jalgebra.2013.09.010}
}

@article{Nishida,
  author     = {Nishida, Goro},
  title      = {The nilpotency of elements of the stable homotopy groups of
                spheres},
  journal    = {J. Math. Soc. Japan},
  fjournal   = {Journal of the Mathematical Society of Japan},
  volume     = {25},
  year       = {1973},
  pages      = {707--732},
  issn       = {0025-5645,1881-1167},
  mrclass    = {55E45},
  mrnumber   = {341485},
  mrreviewer = {J.\ F.\ Adams},
  doi        = {10.2969/jmsj/02540707},
  url        = {https://doi.org/10.2969/jmsj/02540707}
}

@misc{strickland:2012,
  title         = {Tambara functors},
  author        = {Neil Strickland},
  year          = {2012},
  eprint        = {1205.2516},
  archiveprefix = {arXiv},
  primaryclass  = {math.AT},
  url           = {https://arxiv.org/abs/1205.2516}
}

@article{tambara:1993,
  author     = {Tambara, D.},
  title      = {On multiplicative transfer},
  journal    = {Comm. Algebra},
  fjournal   = {Communications in Algebra},
  volume     = {21},
  year       = {1993},
  number     = {4},
  pages      = {1393--1420},
  issn       = {0092-7872,1532-4125},
  mrclass    = {19A22 (18A25 20C99 20J06)},
  mrnumber   = {1209937},
  mrreviewer = {Jacques\ Th\'evenaz},
  doi        = {10.1080/00927879308824627},
  url        = {https://doi.org/10.1080/00927879308824627}
}
\end{document}